\documentclass{article}

\usepackage{cmbright}
\usepackage{amssymb,amsmath,amsthm}
\usepackage{mathrsfs}
\usepackage{epsfig}
\usepackage[usenames,dvipsnames]{color}
\usepackage[colorlinks,allcolors=blue]{hyperref}

\def\B{\mathscr B}
\def\C{\mathbb C}
\def\d{\mathrm{d}}

\def\H{\mathcal H}

\def\N{\mathbb N}

\def\R{\mathbb R}

\def\dom{\mathcal D}

\def\e{\mathop{\mathrm{e}}\nolimits}

\def\im{\mathop{\mathrm{Im}}\nolimits}

\def\linf{\mathrm{L}^\infty}
\def\ltwo{\mathrm{L}^2}
\def\lone{\mathrm{L}^1}
\DeclareMathOperator*{\slim}{s\;\!-lim\;\!}

\newtheorem{Theorem}{Theorem}[section]
\newtheorem{Remark}[Theorem]{Remark}
\newtheorem{Lemma}[Theorem]{Lemma}
\newtheorem{Assumption}[Theorem]{Assumption}

\newtheorem{Proposition}[Theorem]{Proposition}

\begin{document}


\title{Spectral properties of horocycle flows for surfaces of constant negative
curvature}

\author{R. Tiedra de Aldecoa\footnote{Supported by the Chilean Fondecyt Grant 1130168
and by the Iniciativa Cientifica Milenio ICM RC120002 ``Mathematical Physics'' from
the Chilean Ministry of Economy.}}

\date{\small}
\maketitle
\vspace{-1cm}

\begin{quote}
\emph{
\begin{itemize}
\item[] Facultad de Matem\'aticas, Pontificia Universidad Cat\'olica de Chile,\\
Av. Vicu\~na Mackenna 4860, Santiago, Chile
\item[] \emph{E-mail:} rtiedra@mat.puc.cl
\end{itemize}
}
\end{quote}


\begin{abstract}
We consider flows, called $W^{\rm u}$ flows, whose orbits are the unstable manifolds
of a codimension one Anosov flow. Under some regularity assumptions, we give a short
proof of the strong mixing property of $W^{\rm u}$ flows and we show that $W^{\rm u}$
flows have purely absolutely continuous spectrum in the orthocomplement of the
constant functions. As an application, we obtain that time changes of the classical
horocycle flows for compact surfaces of constant negative curvature have purely
absolutely continuous spectrum in the orthocomplement of the constant functions for
time changes in a regularity class slightly less than $C^2$. This generalises recent
results on time changes of horocycle flows.
\end{abstract}


\smallskip


\section{Introduction}
\setcounter{equation}{0}

Horocycle flows for compact surfaces of constant negative curvature, and their
generalisations, are a classical object of study in dynamical systems. Since the
papers of G. A. Hedlund \cite{Hed36,Hed39} and E. Hopf \cite{Hop39} in the 1930's,
many properties of these flows have been put into evidence (far too many to be listed
here). As for spectral properties, in particular the structure of the essential
spectrum, not that many results are available.

In 1953, O. S. Parasyuk \cite{Par53} has shown that the classical horocycle flows for
compact surfaces of constant negative curvature have countable Lebesgue spectrum. In
1974, A. G. Kushnirenko \cite{Kus74} has shown that some small time changes of the
classical horocycle flows for compact surfaces of constant negative curvature are
mixing, and thus have purely continuous spectrum in the orthocomplement of the
constant functions. In 1977, B. Marcus \cite{Mar77} has shown that a large class of
minimal $W^{\rm u}$ flows associated with codimension one Anosov flows (in particular,
a large class of reparametrisations of the classical horocycle flows for compact
surfaces of negative curvature) are mixing. Finally, more recently, G. Forni and C.
Ulcigrai \cite{FU12}, and the author \cite{Tie12,Tie15} have shown that sufficiently
smooth time changes of the classical horocycle flows for compact surfaces of constant
negative curvature have purely absolutely continuous spectrum in the orthocomplement
of the constant functions (in \cite{FU12} Lebesgue spectrum is also obtained, and in
\cite{Tie12} surfaces of finite volume are also considered).

The purpose of this paper is to extend these last results on the absolutely continuous
spectrum of time changes of horocycle flows to a very general class of time changes,
namely, time changes in a regularity class slightly less than $C^2$. Our approach is
the following. We consider as B. Marcus a continuous minimal $W^{\rm u}$ flow, that
is, a continuous minimal flow whose orbits are the unstable manifolds of a
$C^{1+\varepsilon}$ codimension one Anosov flow on a compact connected Riemannian
manifold (minimality is required for the $W^{\rm u}$ flow to be uniquely ergodic).
Under some regularity assumptions on the $W^{\rm u}$ flow and the Anosov flow, we give
a short proof of the strong mixing property of the $W^{\rm u}$ flow. Then, under some
additional regularity assumption, we construct a self-adjoint operator, called
conjugate operator, satisfying a suitable positive commutator estimate with the
self-adjoint generator of the $W^{\rm u}$ flow. Finally, we deduce from this positive
commutator estimate and commutator methods that the generator of the $W^{\rm u}$ flow
has purely absolutely continuous spectrum in the orthocomplement of the constant
functions. As an application, we obtain that time changes of the classical
horocycle flows for compact surfaces of constant negative curvature have purely
absolutely continuous spectrum in the orthocomplement of the constant functions for
time changes in a regularity class slightly less than $C^2$.

Here is a description of the content of the paper. In Section \ref{Sec_Mixing},
we recall some definitions and results on codimension one Anosov flows and minimal
$W^{\rm u}$ flows, we introduce our regularity assumptions (Assumptions \ref{ass_1}
and \ref{ass_2}), and we give a proof of the strong mixing property of the $W^{\rm u}$
flow (Theorem \ref{Theorem_mixing}). In Section \ref{Sec_spectrum}, we recall the
needed facts on commutators of operators and regularity classes associated with them.
Then, under some additional regularity assumption (Assumption \ref{ass_3}), we
construct the conjugate operator (Proposition \ref{Proposition_conjugate}), we prove
the positive commutator estimate (Proposition \ref{Proposition_Mourre}), and we show
that  the generator of the $W^{\rm u}$ flow has purely absolutely continuous spectrum
in the orthocomplement of the constant functions (Theorem \ref{Thm_spectrum}).
Finally, we present the application of this result to time changes of the classical
horocycle flows for compact surfaces of constant negative curvature (Remark
\ref{rem_final}).\\

\noindent
{\bf Acknowledgements.} The author thanks D. Krej{\v{c}}i{\v{r}}{\'{\i}}k for
interesting discussions and for his warm hospitality at the Department of Theoretical
Physics of the Nuclear Physics Institute in \v Re\v z in January 2015. The author also
thanks an anonymous referee for various useful remarks which helped reformulate some
results of this manuscript. 

\section{Strong mixing}\label{Sec_Mixing}
\setcounter{equation}{0}

In this section, we recall some definitions and results on codimension one Anosov
flows and minimal $W^{\rm u}$ flows, and we present a short proof of the strong
mixing property of a class of minimal $W^{\rm u}$ flows. We follow fairly closely
the presentation and notations of B. Marcus \cite{Mar77}, but we refer to the review
papers \cite{Ano69,Mat95,Pes81} for additional information.

A $C^{1+\varepsilon}$ Anosov flow on a compact connected Riemannian manifold $M$ with
distance $d:M\times M\to[0,\infty)$ is a $C^{1+\varepsilon}$ flow $\{f_t\}_{t\in\R}$
with $\varepsilon>0$, without fixed points, satisfying the following property: through
each point $x\in M$ pass three submanifolds $W^{\rm u}(x)$, $W^{\rm s}(x)$, and
$\hbox{Orb}(x)$ whose tangent spaces $E^{\rm u}_x$, $E^{\rm s}_x$, and $E_x$
(respectively) vary continuously with $x\in M$ and satisfy
$$
T_xM=E^{\rm u}_x\oplus E^{\rm s}_x\oplus E_x.
$$
The regularity assumption on $\{f_t\}_{t\in\R}$ means that the function
$\R\times M\ni(t,x)\mapsto f_t(x)\in M$ is of class $C^{1+\varepsilon}$. So,
$\{f_t\}_{t\in\R}$ has a vector field $X_f$ which is H\"older continuous with exponent
$\varepsilon$. The submanifolds $W^{\rm u}(x)$, $W^{\rm s}(x)$, and $\hbox{Orb}(x)$
are called unstable manifolds, stable manifolds, and orbits (respectively), and are
characterised by
\begin{align*}
&W^{\rm u}(x)=\Big\{y\in M\mid\lim_{t\to-\infty}d\big(f_t(x),f_t(y)\big)=0\Big\},\\
&W^{\rm s}(x)=\Big\{y\in M\mid\lim_{t\to+\infty}d\big(f_t(x),f_t(y)\big)=0\Big\},\\
&\hbox{Orb}(x)=\big\{f_t(x)\mid t\in\R\big\}.
\end{align*}
The following two facts are well-known \cite{Ano69}:
\begin{enumerate}
\item[(i)] The families $\{W^{\rm u}(x)\}_{x\in M}$ and $\{W^{\rm s}(x)\}_{x\in M}$
each form a continuous foliation of $M$ (that is, if $y\in W^{\rm u}(x)$ then
$W^{\rm u}(y)=W^{\rm u}(x)$, and $W^{\rm u}(x)$ varies locally continuously with
$x\in M$).
\item[(ii)] $f_t\big(W^{\rm u}(x)\big)=W^{\rm u}\big(f_t(x)\big)$ and
$f_t\big(W^{\rm s}(x)\big)=W^{\rm s}\big(f_t(x)\big)$ for all $t\in\R$ and $x\in M$.
\end{enumerate}

In this paper, we consider a codimension one Anosov flow $\{f_t\}_{t\in\R}$ such that
$\{W^{\rm u}(x)\}_{x\in M}$ is a one-dimensional orientable continuous foliation of
$M$ which defines a continuous minimal flow $\{\phi_s\}_{s\in\R}$ whose orbits
are the unstable manifolds. Such a flow $\{\phi_s\}_{s\in\R}$ is called a minimal
$W^{\rm u}$ flow or a minimal $W^{\rm u}$ parametrisation, and it is uniquely ergodic
with respect to a Borel probability mesure $\mu$ on $M$ \cite{BM77,Mar75}. However,
the unique invariant measure $\mu$ is in general not absolutely continuous with
respect to the volume element \cite[Sec.~6]{Mar77} \cite{Mar70}.

Since $f_t\big(W^{\rm u}(x)\big)=W^{\rm u}\big(f_t(x)\big)$ for all $t\in\R$ and
$x\in M$, there exists a function $s^*:\R\times\R\times M\to\R$ such that
\begin{equation}\label{eq_commutation}
\big(f_t\circ\phi_s\big)(x)=\big(\phi_{s^*(t,s,x)}\circ f_t\big)(x)
\quad\hbox{for all $s,t\in\R$ and $x\in M$.}
\end{equation}
This commutation relation, which describes how the Anosov flow $\{f_t\}_{t\in\R}$
expands $W^{\rm u}$ orbits, is the starting point of our analysis. It generalises the
well-known commutation relation \cite[Rem.~IV.1.2]{BM00} between the geodesic flow and
the classical horocycle flow on the unit tangent bundle of compact orientable surfaces
of constant negative curvature. We recall three facts in relation with
\eqref{eq_commutation}:
\begin{enumerate}
\item[(iii)] The family $\{W^{\rm u}(x)\}_{x\in M}$ admits a uniformly expanding
parametrisation, that is, a continuous parametrisation $\{\widetilde\phi_s\}_{s\in\R}$
such that $f_t\circ\widetilde\phi_s=\widetilde\phi_{\lambda^ts}\circ f_t$ for some
constant $\lambda>1$ and all $s,t\in\R$ \cite[Rem.~1.8 \& Prop.~2.1]{Mar75} (the
constant $\lambda$ is equal to $\e^{h(f_1)}$, with $h(f_1)$ the topological entropy of
$f_1$).
\item[(iv)] Since $\{\widetilde\phi_s\}_{s\in\R}$ is a continuous reparametrisation of
$\{\phi_s\}_{s\in\R}$, and since $\{\phi_s\}_{s\in\R}$ is uniquely ergodic with
respect to the measure $\mu$, the flow $\{\widetilde\phi_s\}_{s\in\R}$ is uniquely
ergodic with respect to a measure $\widetilde\mu$ given in terms of $\mu$ (however,
the measure $\widetilde\mu$ is in general not absolutely continuous with respect to
the measure $\mu$ \cite[\S~3~\&~4]{BS40}).
\item[(v)] The measure $\widetilde\mu$ is invariant under the Anosov flow
$\{f_t\}_{t\in\R}$ \cite[Rem.~6.4]{Mar77}.
\end{enumerate}

In order to be able to define a self-adjoint generator for the flow
$\{\phi_s\}_{s\in\R}$ and to have a simple relation between the measures $\mu$ and
$\widetilde\mu$, we assume the following regularity condition:

\begin{Assumption}\label{ass_1}
The flow $\{\phi_s\}_{s\in\R}$ is of class $C^1$, and $\{\widetilde\phi_s\}_{s\in\R}$
is a $C^1$ reparametristation of $\{\phi_s\}_{s\in\R}$.
\end{Assumption}

Under Assumption \ref{ass_1}, the flow $\{\phi_s\}_{s\in\R}$ has a continuous vector
field $X_\phi$, and there exists a function $\tau\in C^1(M\times\R;\R)$ such that
$\phi_s(x)=\widetilde\phi_{\tau(x,s)}(x)$, $\tau(x,0)=0$, $\tau(x,\;\!\cdot\;\!)$ is
strictly increasing and $\tau(x,s+t)=\tau(x,s)+\tau\big(\phi_s(x),t\big)$ for all
$s,t\in\R$ and $x\in M$ \cite[\S~1]{BS40} (the function $\tau(x,\;\!\cdot\;\!)$ can be
chosen strictly increasing for all $x\in M$ because $M$ is connected). These
properties imply in particular that $\big(\partial_2\tau\big)(x,0)>0$ for all
$x\in M$. Therefore, the function
$$
\rho:M\to\R,~~x\mapsto\frac1{(\partial_2\tau)(x,0)},
$$
is well-defined and belongs to $C\big(M;(0,\infty)\big)$, the flow
$\{\widetilde\phi_s\}_{s\in\R}$ has continuous vector field $\rho X_\phi$, and the
measure $\widetilde\mu$ satisfies \cite[Prop.~3]{Hum74}:
\begin{equation}\label{eq_mu_tilde}
\widetilde\mu=\mu/\widetilde\rho
\quad\hbox{with}\quad
\widetilde\rho:=\rho\int_M\d\mu\,\rho^{-1}.
\end{equation}
Also, one verifies that the pullback operators in $\H:=\ltwo(M,\mu)$ associated with
the flow $\{\phi_s\}_{s\in\R}$,
$$
U^\phi_s\varphi:=\varphi\circ\phi_s,\quad s\in\R,~\varphi\in\H,
$$
define a strongly continuous $1$-parameter group of unitary operators with
$U^\phi_s\hspace{1pt}C^1(M)\subset C^1(M)$ for all $s\in\R$. Thus, Nelson's criterion
\cite[Thm.~VIII.10]{RS80} implies that the generator of the group
$\{U^\phi_s\}_{s\in\R}$,
$$
H_\phi\varphi:=\slim_{s\to0}is^{-1}\big(U^\phi_s-1\big)\varphi,
\quad\varphi\in\dom(H_\phi):=\left\{\varphi\in\H\mid
\lim_{s\to0}|s|^{-1}\big\|\big(U^\phi_s-1\big)\varphi\big\|<\infty\right\},
$$
is essentially self-adjoint on $C^1(M)$ and given by
$$
H_\phi\varphi=iX_\phi\varphi,
\quad\varphi\in C^1(M).
$$
On another hand, the pullback operators associated with the flow $\{f_t\}_{t\in\R}$,
$$
U^f_t\varphi:=\varphi\circ f_t,
\quad t\in\R,~\varphi\in\H,
$$
are not unitary if $\rho\not\equiv1$, but they define a strongly continuous
$1$-parameter group of bounded operators:

\begin{Lemma}\label{Lemma_strong}
Suppose that Assumption \ref{ass_1} is satisfied. Then, $U^f_t\in\B(\H)$ for all 
$t\in\R$, $U^f_sU^f_t=U^f_{s+t}$ for all $s,t\in\R$, $U^f_0=1$, and
$\lim_{\varepsilon\to0}\big\|\big(U^f_{t+\varepsilon}-U^f_t\big)\varphi\big\|=0$ for
all $t\in\R$ and $\varphi\in\H$.
\end{Lemma}

\begin{proof}
A direct calculation using \eqref{eq_mu_tilde} and the fact that $\widetilde\mu$ is
invariant under $\{f_t\}_{t\in\R}$ implies for $t\in\R$ and $\varphi\in\H$ that
$$
\big\|U^f_t\varphi\big\|^2
=\int_M\d\widetilde\mu\,\widetilde\rho\;\!\big|\varphi\circ f_t\big|^2
=\int_M\d\widetilde\mu\,\big(\widetilde\rho\circ f_{-t}\big)|\varphi|^2
=\int_M\d\mu\,\frac{\rho\circ f_{-t}}\rho\;\!|\varphi|^2\\
\le\frac{\max(\rho)}{\min(\rho)}\;\!\|\varphi\|^2.
$$
Thus, $U^f_t\in\B(\H)$ with
\begin{equation}\label{eq_norm_U^f}
\big\|U^f_t\big\|\le\sqrt{\frac{\max(\rho)}{\min(\rho)}}.
\end{equation}
The group properties $U^f_sU^f_t=U^f_{s+t}$ for $s,t\in\R$ and $U^f_0=1$ are evident.
To show the last property, take $t\in\R$ and $\varphi\in C(M)$. Then, the continuity
of $\big(\varphi\circ f_{t+\varepsilon}-\varphi\circ f_t\big)$ and Lebesgue dominated
convergence theorem imply that
$$
\lim_{\varepsilon\to0}\big\|\big(U^f_{t+\varepsilon}-U^f_t\big)\varphi\big\|^2
=\int_M\d\mu\,\lim_{\varepsilon\to0}
\big|\varphi\circ f_{t+\varepsilon}-\varphi\circ f_t\big|^2
=0.
$$
Since $C(M)$ is dense in $\H$ and $\big(U^f_{t+\varepsilon}-U^f_t\big)\in\B(\H)$, this
implies that
$\lim_{\varepsilon\to0}\big\|\big(U^f_{t+\varepsilon}-U^f_t\big)\varphi\big\|=0$ for
all $t\in\R$ and $\varphi\in\H$.
\end{proof}

For the next lemma, we need the following result of B. Marcus:

\begin{Lemma}[Lemma 3.1 of \cite{Mar77}]\label{Lemma_Marcus}
Let $\{\phi_s\}_{s\in\R}$ be a minimal $W^{\rm u}$ flow on $M$. Then,
$$
\lim_{s\to\infty}s^{-1}s^*(t,s,x)=\lambda^t
$$
uniformly in $x\in M$ and $t$ in a given compact interval of $\R$.
\end{Lemma}

We also need to introduce as B. Marcus in \cite[Sec.~4]{Mar77} the following
regularity condition on the function $s^*$:

\begin{Assumption}\label{ass_2}
The derivative
$$
u_{t,s}(x):=\big(\partial_1\partial_2s^*\big)(t,s,x)
$$
exists and is continuous in $s,t\in\R$ and $x\in M$.
\end{Assumption}

\begin{Lemma}\label{Lemma_limit}
Suppose that Assumption \ref{ass_2} is satisfied. Then,
\begin{equation}\label{eq_int}
\lim_{s\to\infty}s^{-1}\big(\partial_1s^*\big)(t,s,x)
=\int_M\d\mu\,u_{t,0}
=\ln(\lambda)\hspace{1pt}\lambda^t
\end{equation}
uniformly in $x\in M$ and $t$ in a given compact interval of $\R$.
\end{Lemma}

\begin{proof}
Let $I\subset\R$ be a compact interval, and take $r,s\in\R$, $t\in I$ and $x\in M$.
Using successively the relation $\big(\partial_1s^*\big)(t,0,x)=0$, Assumption
\ref{ass_2} and the cocycle equation
\begin{equation}\label{eq_cocycle}
s^*(t,r+s,x)=s^*(t,r,x)+s^*\big(t,s,\phi_r(x)\big),
\end{equation}
we obtain
$$
\big(\partial_1s^*\big)(t,s,x)
=\int_0^s\d r\,\frac\d{\d r}\;\!\big(\partial_1s^*\big)(t,r,x)
=\int_0^s\d r\,\frac\d{\d t}\;\!\frac\d{\d s}\Big|_{s=0}\;\!s^*(t,r+s,x)
=\int_0^s\d r\,u_{t,0}\big(\phi_r(x)\big).
$$
Therefore, it follows by the unique ergodicity of $\{\phi_s\}_{s\in\R}$ that
\begin{equation}\label{eq_prem}
\lim_{s\to\infty}s^{-1}\big(\partial_1s^*\big)(t,s,x)
=\lim_{s\to\infty}s^{-1}\int_0^s\d r\,u_{t,0}\big(\phi_r(x)\big)
=\int_M\d\mu\,u_{t,0}
\end{equation}
uniformly in $x\in M$ and $t\in I$  (the uniformity in $t$ follows from the continuity
of $u_{t,0}$ in $t$). To show the second equality in \eqref{eq_int}, we note from
Lemma \ref{Lemma_Marcus} that
$$
\lim_{s\to\infty}s^{-1}s^*(t,s,x)=\lambda^t
$$
uniformly in $x\in M$ and $t\in I$. We also note from \eqref{eq_prem} that
$s^{-1}\big(\partial_1s^*\big)(t,s,x)$ converges to some function of $t$ uniformly in
$x\in M$ and $t\in I$. So, it follows by a uniform convergence argument that
$$
\lim_{s\to\infty}s^{-1}\big(\partial_1s^*\big)(t,s,x)
=\frac\d{\d t}\;\!\lim_{s\to\infty}s^{-1}s^*(t,s,x)
=\frac\d{\d t}\;\!\lambda^t
=\ln(\lambda)\hspace{1pt}\lambda^t
$$
uniformly in $x\in M$ and $t\in I$.
\end{proof}

In the following theorem, we give a short proof, inspired by \cite[Thm.~4.1]{Tie15_2},
of the strong mixing property of the flow $\{\phi_s\}_{s\in\R}$. It can be viewed as a
simplified version of the proof of B. Marcus \cite[Sec.~4]{Mar77} in the case 
Assumption \ref{ass_1} is satisfied (the proof of B. Marcus works without Assumption
\ref{ass_1}).

\begin{Theorem}[Strong mixing]\label{Theorem_mixing}
Suppose that Assumptions \ref{ass_1} and \ref{ass_2} are satisfied. Then,
$\lim_{s\to\infty}\big\langle\psi,U^\phi_s\varphi\big\rangle=0$ for all $\psi\in\H$
and $\varphi\in\ker(H_\phi)^\perp$. In particular, the flow $\{\phi_s\}_{s\in\R}$ is
strongly mixing with respect to the measure $\mu$.
\end{Theorem}

\begin{proof}
Take $\varphi\in C^1(M)$, $s,t\in\R$ and $x\in M$. Then, we know from
\eqref{eq_commutation} that
$$
\big(\varphi\circ f_t\circ\phi_s\big)(x)
=\big(\varphi\circ\phi_{s^*(t,s,x)}\circ f_t\big)(x).
$$
Applying the derivative $\frac\d{\d t}\big|_{t=0}$ and using the relation
$s^*(0,s,x)=s$, we obtain
\begin{align*}
&\big(X_f\varphi\big)\big(\phi_s(x)\big)\\
&=\frac\d{\d t}\Big|_{t=0}\big(\varphi\circ\phi_{s^*(t,s,x)}\circ f_t\big)(x)\\
&=\lim_{\varepsilon\to0}\varepsilon^{-1}
\big\{\big(\varphi\circ\phi_{s^*(\varepsilon,s,x)}\circ f_\varepsilon\big)(x)
-\big(\varphi\circ\phi_{s^*(0,s,x)}\circ f_\varepsilon\big)(x)\big\}\\
&\quad+\lim_{\varepsilon\to0}\varepsilon^{-1}
\big\{\big(\varphi\circ\phi_s\circ f_\varepsilon\big)(x)
-\big(\varphi\circ\phi_s\circ f_0\big)(x)\big\}\\
&=\big(\partial_1s^*\big)(0,s,x)\cdot\big(X_\phi\varphi\big)\big(\phi_s(x)\big)
+\big(X_f(\varphi\circ\phi_s)\big)(x),
\end{align*}
which is equivalent to
\begin{equation}\label{eq_conj}
\big(\partial_1s^*\big)(0,s,\;\!\cdot\,)\;\!U^\phi_sH_\phi\varphi
=iU^\phi_sX_f\varphi-iX_fU^\phi_s\varphi.
\end{equation}
So, for $\psi\in C^1(M)$, we infer from Lemma \ref{Lemma_limit} that
\begin{align*}
&\lim_{s\to\infty}\big\langle\rho^{-1}\psi,U^\phi_sH_\phi\varphi\big\rangle\\
&=\ln(\lambda)^{-1}\lim_{s\to\infty}
\big\langle\rho^{-1}\psi,\ln(\lambda)U^\phi_sH_\phi\varphi\big\rangle\\
&=\ln(\lambda)^{-1}\lim_{s\to\infty}\big\langle\rho^{-1}\psi,
s^{-1}\big(\partial_1s^*\big)(0,s,\;\!\cdot\,)\;\!U^\phi_sH_\phi\varphi\big\rangle\\
&=\ln(\lambda)^{-1}\lim_{s\to\infty}s^{-1}
\big\langle\rho^{-1}\psi,iU^\phi_sX_f\varphi\big\rangle
-\ln(\lambda)^{-1}\lim_{s\to\infty}s^{-1}
\big\langle\rho^{-1}\psi,iX_fU^\phi_s\varphi\big\rangle\\
&=0-\ln(\lambda)^{-1}\lim_{s\to\infty}s^{-1}
\big\langle\psi,i\rho^{-1}X_fU^\phi_s\varphi\big\rangle.
\end{align*}
But, the operator $i\rho^{-1}X_f$ is symmetric on $C^1(M)$. So,
$$
\lim_{s\to\infty}s^{-1}\big\langle\psi,i\rho^{-1}X_fU^\phi_s\varphi\big\rangle
=\lim_{s\to\infty}s^{-1}\big\langle i\rho^{-1}X_f\psi,U^\phi_s\varphi\big\rangle
=0,
$$
and thus
\begin{equation}\label{eq_strong_1}
\lim_{s\to\infty}\big\langle\rho^{-1}\psi,U^\phi_sH_\phi\varphi\big\rangle=0.
\end{equation}
Moreover, the set $\rho^{-1}C^1(M)$ is dense in $\H$ because $C^1(M)$ is dense in $\H$
and $\rho^{-1}:\H\to\H$ is an homeomorphism, and the set $H_\phi C^1(M)$ is dense in
$\ker(H_\phi)^\perp$ because $C^1(M)$ is a core for $H_\phi$ and $H_\phi\dom(H_\phi)$
is dense in $\ker(H_\phi)^\perp$. Therefore, \eqref{eq_strong_1} implies that
$\lim_{s\to\infty}\big\langle\psi,U^\phi_s\varphi\big\rangle=0$ for all $\psi\in\H$
and $\varphi\in\ker(H_\phi)^\perp$.
\end{proof}

Under Assumptions \ref{ass_1} and \ref{ass_2}, the result of Theorem
\ref{Theorem_mixing} applies in particular to the case of reparametrisations of
classical horocycle flows on the unit tangent bundle of compact connected orientable
surfaces of constant negative curvature (see \cite[Cor.~4.2]{Mar77}).

\section{Absolutely continuous spectrum}\label{Sec_spectrum}
\setcounter{equation}{0}

We know from Theorem \ref{Theorem_mixing} that, under Assumptions \ref{ass_1} and
\ref{ass_2}, the flow $\{\phi_s\}_{s\in\R}$ is strongly mixing with respect to the
measure $\mu$. Therefore, its generator $H_\phi$ has purely continuous spectrum in
$\R\setminus\{0\}$. Our goal in this section is to show that the spectrum of $H_\phi$
is even {\em purely absolutely continuous} in $\R\setminus\{0\}$ under some additional
regularity assumption. For this, we first need to recall some results on commutator
methods borrowed from \cite{ABG96,Sah97_2} (see also the original paper \cite{Mou81}
of \'E. Mourre).

\subsection{Commutators and regularity classes}\label{Sec_Commutators}

Let $\H$ be a Hilbert space with scalar product
$\langle\;\!\cdot\;\!,\;\!\cdot\;\!\rangle$ antilinear in the first argument, denote
by $\B(\H)$ the set of bounded linear operators on $\H$, and write $\|\;\!\cdot\;\!\|$
both for the norm on $\H$ and the norm on $\B(\H)$. Let $A$ be a self-adjoint operator
in $\H$ with domain $\dom(A)$, and take $S\in\B(\H)$. For any $k\in\N$, we say that
$S$ belongs to $C^k(A)$, with notation $S\in C^k(A)$, if the map
\begin{equation}\label{eq_group}
\R\ni t\mapsto\e^{-itA}S\e^{itA}\in\B(\H)
\end{equation}
is strongly of class $C^k$. In the case $k=1$, one has $S\in C^1(A)$ if and only if
the quadratic form
$$
\dom(A)\ni\varphi\mapsto\big\langle\varphi,SA\hspace{1pt}\varphi\big\rangle
-\big\langle A\hspace{1pt}\varphi,S\varphi\big\rangle\in\C
$$
is continuous for the topology induced by $\H$ on $\dom(A)$. We denote by $[S,A]$ the
bounded operator associated with the continuous extension of this form, or
equivalently $-i$ times the strong derivative of the function \eqref{eq_group} at
$t=0$.

If $H$ is a self-adjoint operator in $\H$ with domain $\dom(H)$ and spectrum
$\sigma(H)$, we say that $H$ is of class $C^k(A)$ if $(H-z)^{-1}\in C^k(A)$ for some
$z\in\C\setminus\sigma(H)$. So, $H$ is of class $C^1(A)$ if and only if the quadratic
form
$$
\dom(A)\ni\varphi\mapsto
\big\langle\varphi,(H-z)^{-1}A\hspace{1pt}\varphi\big\rangle
-\big\langle A\hspace{1pt}\varphi,(H-z)^{-1}\varphi\big\rangle\in\C
$$
extends continuously to a bounded form defined by the operator
$[(H-z)^{-1},A]\in\B(\H)$. In such a case, the set $\dom(H)\cap\dom(A)$ is a core for
$H$ and the quadratic form
$$
\dom(H)\cap\dom(A)\ni\varphi\mapsto\big\langle H\varphi,A\hspace{1pt}\varphi\big\rangle
-\big\langle A\hspace{1pt}\varphi,H\varphi\big\rangle\in\C
$$
is continuous in the topology of $\dom(H)$ \cite[Thm.~6.2.10(b)]{ABG96}. This form
then extends uniquely to a continuous quadratic form on $\dom(H)$ which can be
identified with a continuous operator $[H,A]$ from $\dom(H)$ to the adjoint space
$\dom(H)^*$. In addition, the following relation holds in $\B(\H)$:
\begin{equation}\label{eq_resolvent}
\big[(H-z)^{-1},A\big]=-(H-z)^{-1}[H,A](H-z)^{-1}.
\end{equation}

Let $E^H(\hspace{1pt}\cdot\hspace{1pt})$ denote the spectral measure of the
self-adjoint operator $H$, and assume that $H$ is of class $C^1(A)$. If there exist a
Borel set $I\subset\R$, a number $a>0$ and a compact operator $K\in\B(\H)$ such that
\begin{equation}\label{Mourre_H}
E^H(I)[\hspace{1pt}iH,A]E^H(I)\ge aE^H(I)+K,
\end{equation}
then one says that $H$ satisfies a Mourre estimate on $I$ and that $A$ is a conjugate
operator for $H$ on $I$. Also, one says that $H$ satisfies a strict Mourre estimate on
$I$ if \eqref{Mourre_H} holds with $K=0$. One of the consequences of a Mourre estimate
is to imply spectral results for $H$ on $I$. We recall here these spectral results in
the case where $H$ is of class $C^2(A)$ (see \cite[Sec.~7.1.2]{ABG96} and
\cite[Thm.~0.1]{Sah97_2} for more details).

\begin{Theorem}\label{thm_absolute}
Let $H$ and $A$ be self-ajoint operators in a Hilbert space $\H$, with $H$ of class
$C^2(A)$. Suppose there exist a bounded Borel set $I\subset\R$, a number $a>0$ and a
compact operator $K\in\B(\H)$ such that
\begin{equation}\label{Mourre_H_bis}
E^H(I)[\hspace{1pt}iH,A]E^H(I)\ge aE^H(I)+K.
\end{equation}
Then, $H$ has at most finitely many eigenvalues in $I$, each one of finite
multiplicity, and $H$ has no singular continuous spectrum in $I$. Furthermore, if
\eqref{Mourre_H_bis} holds with $K=0$, then $H$ has only purely absolutely continuous
spectrum in $I$ (no singular spectrum).
\end{Theorem}

\subsection{Absolutely continuous spectrum}\label{Sec_Absolute}

We show in this section that the spectrum of $H_\phi$ is purely absolutely continuous
in $\R\setminus\{0\}$ under some additional regularity assumption. We start with two
technical lemmas.

\begin{Lemma}\label{Lemma_form_u}
Suppose that Assumptions \ref{ass_1} and \ref{ass_2} are satisfied.
\begin{enumerate}
\item[(a)] If $s\in\R$ and $\varphi\in C^1(M)$, then
$
U^\phi_sX_fU^\phi_{-s}\varphi
=X_f\varphi+\int_0^s\d r\,\big(u_{0,0}\circ\phi_r\big)X_\phi\varphi
$.
\item[(b)] If $z\in\C\setminus\R$ and $\varphi\in C^1(M)$, then
$X_f\big(H_\phi-z\big)^{-1}\varphi\in\H$.
\item[(c)] If $X_f(\rho)\in C(M)$, then $u_{0,0}=\ln(\lambda)+\rho^{-1}X_f(\rho)$.
\end{enumerate}
\end{Lemma}

\begin{proof}
(a) Take $s\in\R$ and $\varphi\in C^1(M)$. Then, \eqref{eq_conj}, the identity
$\big(\partial_1s^*\big)(0,0,\;\!\cdot\,)\equiv0$ and Assumption \ref{ass_2} imply
\begin{equation}\label{eq_integral}
U^\phi_sX_fU^\phi_{-s}\varphi
=X_f\varphi+\big(\partial_1s^*\big)(0,s,\;\!\cdot\,)X_\phi\varphi
=X_f\varphi+\int_0^s\d r\,u_{0,r}X_\phi\varphi.
\end{equation}
On another hand, the cocycle equation \eqref{eq_cocycle} implies for all $r,s,t\in\R$
that
$$
s^*(t,s,\;\!\cdot\,)
=s^*\big(t,s-r,\phi_r(\;\!\cdot\;\!)\big)-s^*\big(t,-r,\phi_r(\;\!\cdot\;\!)\big)
=U^\phi_r\big(s^*(t,s-r,\;\!\cdot\,)-s^*\big(t,-r,\;\!\cdot\,)\big)U^\phi_{-r}.
$$
Deriving with respect to $t$ and $s$, we thus obtain
$$
u_{t,s}=U^\phi_ru_{t,s-r}U^\phi_{-r}=u_{t,s-r}\circ\phi_r.
$$
In particular, we have $u_{0,r}=u_{0,0}\circ\phi_r$, and the claim follows from
\eqref{eq_integral}.

(b) We give the proof in the case $\im(z)>0$ since the case $\im(z)<0$ is analogous.
Take $\varphi\in C^1(M)$. Then, the formula
$\big(H_\phi-z\big)^{-1}\varphi=i\int_0^\infty\d r\,\e^{irz}U^\phi_r\varphi$ and point
(a) imply
\begin{align}
&X_f\big(H_\phi-z\big)^{-1}\varphi\nonumber\\
&=i\lim_{t\to0}\int_0^\infty\d r\,\e^{irz}t^{-1}\big(U^f_t-1\big)U^\phi_r\varphi
\nonumber\\
&=i\lim_{t\to0}\int_0^\infty\d r\,\e^{irz}t^{-1}\int_0^t\d s\,U^f_sX_fU^\phi_r\varphi
\nonumber\\
&=i\lim_{t\to0}\int_0^\infty\d r\,\e^{irz}t^{-1}\int_0^t\d s\,U^f_sU^\phi_r
\left(X_f+\int_0^{-r}\d q\,\big(u_{0,0}\circ\phi_q\big)X_\phi\right)\varphi.\label{eq_lebesgue}
\end{align}
Now, we know from \eqref{eq_norm_U^f} that
$\big\|U^f_s\big\|\le\sqrt{\frac{\max(\rho)}{\min(\rho)}}$ for all $s\in\R$. Thus, we
have
\begin{align*}
&\left\|\hspace{1pt}\e^{irz}t^{-1}\int_0^t\d s\,U^f_sU^\phi_r
\left(X_f+\int_0^{-r}\d q\,\big(u_{0,0}\circ\phi_q\big)X_\phi\right)\varphi\right\|\\
&\le\e^{-r\im(z)}\sqrt{\frac{\max(\rho)}{\min(\rho)}}
\;\!\big(\big\|X_f\varphi\big\|+r\hspace{1pt}\|u_{0,0}\|_{\linf(X,\mu)}\big\|
X_\phi\varphi\big\|\big)\\
&\in\lone\big([0,\infty),\d r\big),
\end{align*}
and we can apply Lebesgue dominated convergence theorem to \eqref{eq_lebesgue} to
obtain
\begin{align*}
&X_f\big(H_\phi-z\big)^{-1}\varphi\\
&=i\int_0^\infty\d r\,\e^{irz}\lim_{t\to0}t^{-1}\int_0^t\d s\,U^f_sU^\phi_r
\left(X_f+\int_0^{-r}\d q\,\big(u_{0,0}\circ\phi_q\big)X_\phi\right)\varphi\\
&=i\int_0^\infty\d r\,\e^{irz}U^\phi_r
\left(X_f+\int_0^{-r}\d q\,\big(u_{0,0}\circ\phi_q\big)X_\phi\right)\varphi\\
&\in\H.
\end{align*}

(c) The proof is inspired by a result of L. W. Green in the case of
the classical horocycle flows on the unit tangent bundle of compact connected
orientable surfaces of negative curvature (see \cite[Eq.~(3.3) \& Lemma~3.3]{Gre78}).
Take $\psi,\varphi\in C^1(M)$. Then, the facts that
$\{\widetilde\phi_s\}_{s\in\R}$ has vector field $\rho X_\phi$, that 
$\{\widetilde\phi_s\}_{s\in\R}$ and $\{f_t\}_{t\in\R}$ preserve the measure
$\widetilde\mu=\mu/\widetilde\rho$ and that
$f_t\circ\widetilde\phi_{-s}=\widetilde\phi_{-\lambda^ts}\circ f_t$ imply
\begin{align*}
\big\langle X_\phi\psi,X_f\varphi\big\rangle
&=\frac\d{\d t}\Big|_{t=0}\;\!\frac\d{\d s}\Big|_{s=0}
\big\langle\rho^{-1}\big(\psi\circ{\widetilde\phi_s}\big),
\varphi\circ f_t\big\rangle\\
&=\frac\d{\d t}\Big|_{t=0}\;\!\frac\d{\d s}\Big|_{s=0}
\big\langle\rho^{-1}\psi,\varphi\circ f_t\circ{\widetilde\phi_{-s}}\big\rangle\\
&=\frac\d{\d t}\Big|_{t=0}\;\!\frac\d{\d s}\Big|_{s=0}\big\langle\rho^{-1}\psi,
\varphi\circ\widetilde\phi_{-\lambda^ts}\circ f_t\big\rangle\\
&=-\frac\d{\d t}\Big|_{t=0}
\big\langle\rho^{-1}\psi,\lambda^t\big(\rho X_\phi\varphi\big)\circ f_t\big\rangle\\
&=-\frac\d{\d t}\Big|_{t=0}
\big\langle\psi\circ f_{-t},\lambda^tX_\phi\varphi\big\rangle\\
&=\big\langle X_f\psi,X_\phi\varphi\big\rangle
-\big\langle\psi,\ln(\lambda)X_\phi\varphi\big\rangle.
\end{align*}
On another hand, point (a) and the symmetricity of $i\rho^{-1}X_f$ on $C^1(M)$ imply
\begin{align*}
&\big\langle\psi,u_{0,0}X_\phi\varphi\big\rangle\\
&=\frac\d{\d s}\Big|_{s=0}
\big\langle\rho U^\phi_{-s}\psi,\rho^{-1}X_fU^\phi_{-s}\varphi\big\rangle\\
&=-\big\langle X_\phi\psi,X_f\varphi\big\rangle
+\big\langle\rho^{-1}X_f\rho\hspace{1pt}\psi,X_\phi\varphi\big\rangle\\
&=-\big\langle X_\phi\psi,X_f\varphi\big\rangle
+\big\langle X_f\psi,X_\phi\varphi\big\rangle
+\big\langle\psi,\rho^{-1}X_f(\rho)X_\phi\varphi\big\rangle.
\end{align*}
Combining the two relations, we thus obtain
$$
\big\langle\psi,\big(u_{0,0}-\ln(\lambda)
-\rho^{-1}X_f(\rho)\big)X_\phi\varphi\big\rangle
=0,
$$
and we infer from the density of $C^1(M)$ in $\H$ and the density of $X_\phi C^1(M)$
in $\ker(H_\phi)^\perp$ that
\begin{equation}\label{prod_scal}
\big\langle\psi,\big(u_{0,0}-\ln(\lambda)-\rho^{-1}X_f(\rho)\big)\varphi\big\rangle
=0
\end{equation}
for all $\psi\in\H$ and $\varphi\in\ker(H_\phi)^\perp$. Now, Lemma \ref{Lemma_limit},
the fact that $\{f_t\}_{t\in\R}$ preserves $\widetilde\mu=\mu/\widetilde\rho$, and the
ergodicity of $\{\phi_s\}_{s\in\R}$ imply that
$$
u_{0,0}-\ln(\lambda)-\rho^{-1}X_f(\rho)
\in\left\{\varphi\in\H\mid\int_M\d\mu\,\varphi=0\right\}
=\ker(H_\phi)^\perp.
$$
So, we can set $\psi=1$ and $\varphi=u_{0,0}-\ln(\lambda)-\rho^{-1}X_f(\rho)$ in
\eqref{prod_scal} to get
$$
\int_M\d\mu\,\big(u_{0,0}-\ln(\lambda)-\rho^{-1}X_f(\rho)\big)^2=0,
$$
and then infer that $u_{0,0}-\ln(\lambda)-\rho^{-1}X_f(\rho)\equiv0$ by the continuity
of $u_{0,0}$ and $\rho^{-1}X_f(\rho)$.
\end{proof}

\begin{Lemma}\label{Lemma_essential}
Let $X$ be a $C^1$ vector field on $M$ and $g\in C^1(M;\R)$. Assume that the operator
$$
A\varphi:=i\hspace{0.1pt}(X+g)\varphi,\quad\varphi\in C^1(M),
$$
is symmetric in $\H$. Then, $A$ is essentially self-adjoint in $\H$.
\end{Lemma}

\begin{proof}
Since $X$ is of class $C^1$, $X$ admits a $C^1$ flow $\{\zeta_s\}_{s\in\R}$
\cite[Thm.~3.43]{Irw01}. Thus, the operators
$$
V_s\varphi:=\e^{\int_0^s\d r\,(g\circ\zeta_r)}\varphi\circ\zeta_s,
\quad s\in\R,~\varphi\in C^1(M),
$$
are well-defined operators in $\H$. Simple calculations show that
$V_sV_t\varphi=V_{s+t}\varphi$ and $V_0\varphi=\varphi$ for $s,t\in\R$ and
$\varphi\in C^1(M)$, that $\lim_{\varepsilon\to0}\|(V_{s+\varepsilon}-V_s)\varphi\|=0$
for $s\in\R$ and $\varphi\in C^1(M)$, that $V_s\hspace{1pt}C^1(M)\subset C^1(M)$ for
$s\in\R$, and that $\frac\d{\d s}\hspace{1pt}V_s\varphi=-iAV_s\varphi$ for $s\in\R$
and $\varphi\in C^1(M)$. Furthermore, we have for $s\in\R$ and
$\psi,\varphi\in C^1(M)$ the equalities
$$
\big\langle V_s\psi,V_s\varphi\big\rangle-\langle\psi,\varphi\rangle
=\int_0^s\d r\,\frac\d{\d r}\;\!\big\langle V_r\psi,V_r\varphi\big\rangle
=i\int_0^s\d r\,\big(\big\langle AV_r\psi,V_r\varphi\big\rangle
-\big\langle V_r\psi,AV_r\varphi\big\rangle\big)
=0,
$$
due to the symmetricity of $A$. Therefore, the family $\{V_s\}_{s\in\R}$ satisfies on
$C^1(M)$ the properties of a strongly continuous $1$-parameter group of isometric
operators in $\H$ with $V_s\hspace{1pt}C^1(M)\subset C^1(M)$ for all $s\in\R$, and
with generator equal to $A$ on $C^1(M)$. Since $C^1(M)$ is dense in $\H$, it follows
that $\{V_s\}_{s\in\R}$ extends to a strongly continuous $1$-parameter group of
isometric (and thus unitary) operators in $\H$ with
$V_s\hspace{1pt}C^1(M)\subset C^1(M)$ for all $s\in\R$, and with generator equal to
$A$ on $C^1(M)$. Thus, Nelson's criterion \cite[Thm.~VIII.10]{RS80} implies that $A$
is essentially self-adjoint in $\H$.
\end{proof}

In the rest of the paper, in addition to Assumptions \ref{ass_1} and \ref{ass_2}, we
assume the following:

\begin{Assumption}\label{ass_3}
The vector fields $X_f$ and $X_\phi$ are of class $C^1$, $X_f(\rho)\in C(M)$ and
$\rho^{-1}X_f(\rho)\in C^1(M)$.
\end{Assumption}

The conditions $X_f(\rho)\in C(M)$ and $\rho^{-1}X_f(\rho)\in C^1(M)$ are equivalent
to the condition $X_f\big(\ln(\rho)\big)\in C^1(M)$. So, if one prefers, one can
replace the two conditions $X_f(\rho)\in C(M)$ and $\rho^{-1}X_f(\rho)\in C^1(M)$ in
Assumption \ref{ass_3} by the single condition $X_f\big(\ln(\rho)\big)\in C^1(M)$.

In the next proposition we define and prove the self-adjointness of the conjugate
operator. Intuitively, the conjugate operator is constructed as follows. First, we
take the sum of the vector field $2iX_f$ and its ``divergence'' $i\rho^{-1}X_f(\rho)$
to get a symmetric operator on $C^1(M)$. Then, we take the Birkhoff average of the
resulting operator along the flow $\{\phi_s\}_{s\in\R}$ to take into account the
unique ergodicity of $\{\phi_s\}_{s\in\R}$.

\begin{Proposition}[Conjugate operator]\label{Proposition_conjugate}
Suppose that Assumptions \ref{ass_1}, \ref{ass_2}, and \ref{ass_3} are satisfied.
Then, the operator
$$
A_t\varphi
:=\frac1t\int_0^t\d s\,U^\phi_s\big(2iX_f+i\rho^{-1}X_f(\rho)\big)
U^\phi_{-s}\varphi,\quad t>0,~\varphi\in C^1(M),
$$
is essentially self-adjoint in $\H$ (and its closure is denoted by the same symbol).
\end{Proposition}

\begin{proof}
Since $\rho^{-1}X_f(\rho)=\widetilde\rho\hspace{1pt}{}^{-1}X_f(\widetilde\rho)$,
the operator $\big(2iX_f+i\rho^{-1}X_f(\rho)\big)$ is symmetric on $C^1(M)$.
Therefore, the operator $A_t$ is also symmetric on $C^1(M)$ because
$U^\phi_s\hspace{1pt}C^1(M)\subset C^1(M)$ for all $s\in\R$. To show that $A_t$ is
essentially self-adjoint on $C^1(M)$, we take $\psi,\varphi\in C^1(M)$. Then, Lemma
\ref{Lemma_form_u}(a) implies that
\begin{align}
&\big\langle\psi,A_t\varphi\big\rangle\nonumber\\
&=\left\langle\psi,\frac1t\int_0^t\d s
\left(2iX_f+2i\int_0^s\d r\,\big(u_{0,0}\circ\phi_r\big)X_\phi
+i\big(\rho^{-1}X_f(\rho)\big)\circ\phi_s\right)\varphi\right\rangle\nonumber\\
&=\big\langle\psi,i\big(2X_f+a_tX_\phi+b_t\big)\varphi\big\rangle
\label{eq_A_t}
\end{align}
with
$$
a_t:=\frac2t\int_0^t\d s\int_0^s\d r\,\big(u_{0,0}\circ\phi_r\big)
\quad\hbox{and}\quad
b_t:=\frac1t\int_0^t\d s\,\big(\rho^{-1}X_f(\rho)\big)\circ\phi_s.
$$
Furthermore, Assumption \ref{ass_3}, Lemma \ref{Lemma_form_u}(c) and Leibniz integral
rule imply that $X_f$, $X_\phi$, $a_t$ and $b_t$ are of class $C^1$. Therefore, we can
apply Lemma \ref{Lemma_essential} with $X:=2X_f+a_tX_\phi$ and $g:=b_t$ to conclude
that $A_t$ is essentially self-adjoint in $\H$.
\end{proof}

\begin{Lemma}\label{Lemma_C2}
Suppose that Assumptions \ref{ass_1}, \ref{ass_2}, and \ref{ass_3} are satisfied.
Then, we have for $t>0$
\begin{enumerate}
\item[(a)] $\big(H_\phi-i\big)^{-1}\in C^1(A_t)$
with
$$
\big[i\big(H_\phi-i\big)^{-1},A_t\big]
=2\big(H_\phi-i\big)^{-1}c_tH_\phi\big(H_\phi-i\big)^{-1}
-\big[\big(H_\phi-i\big)^{-1},c_t\big]
\quad\hbox{and}\quad
c_t:=\frac1t\int_0^t\d s\,\big(u_{0,0}\circ\phi_s\big),
$$
\item[(b)] $\big(H_\phi-i\big)^{-1}\in C^2(A_t)$.
\end{enumerate}
\end{Lemma}

\begin{proof}
(a) Set $A:=\big(2iX_f+i\rho^{-1}X_f(\rho)\big)$ on $C^1(M)$ and take
$\varphi\in C^1(M)$. Then, we know from Lemma \ref{Lemma_form_u}(b)-(c) that
\begin{align*}
&\big\langle\big(H_\phi+i\big)^{-1}\varphi,A\varphi\big\rangle
-\big\langle A\varphi,\big(H_\phi-i\big)^{-1}\varphi\big\rangle\\
&=\big\langle\big(2iX_f+i\rho^{-1}X_f(\rho)\big)\big(H_\phi+i\big)^{-1}\varphi,
\big(H_\phi-i\big)^{-1}\big(H_\phi-i\big)\varphi\big\rangle\\
&\quad-\big\langle\big(H_\phi+i\big)^{-1}\big(H_\phi+i\big)\varphi,
\big(2iX_f+i\rho^{-1}X_f(\rho)\big)\big(H_\phi-i\big)^{-1}\varphi\big\rangle\\
&=i\hspace{1pt}\frac\d{\d s}\Big|_{s=0}
\big\{\big\langle\big(2iX_f+i\rho^{-1}X_f(\rho)\big)\big(H_\phi+i\big)^{-1}\varphi,
\big(H_\phi-i\big)^{-1}\big(U^\phi_s-s\big)\varphi\big\rangle\\
&\quad\qquad\qquad-\big\langle\big(H_\phi+i\big)^{-1}\big(U^\phi_{-s}-s\big)\varphi,
\big(2iX_f+i\rho^{-1}X_f(\rho)\big)\big(H_\phi-i\big)^{-1}\varphi\big\rangle\big\}\\
&=-\frac\d{\d s}\Big|_{s=0}
\big\{\big\langle\big(H_\phi+i\big)^{-1}\varphi,
\big(2X_f+u_{0,0}\big)\big(H_\phi-i\big)^{-1}U^\phi_s\varphi\big\rangle\\
&\quad\qquad\qquad-\big\langle\big(H_\phi+i\big)^{-1}\varphi,
U^\phi_s\big(2X_f+u_{0,0}\big)\big(H_\phi-i\big)^{-1}\varphi\big\rangle\big\}\\
&=-2\left\langle\varphi,\frac\d{\d s}\Big|_{s=0}
\big(H_\phi-i\big)^{-1}\big\{X_f\big(H_\phi-i\big)^{-1}U^\phi_s
-U^\phi_sX_f\big(H_\phi-i\big)^{-1}\big\}\varphi\right\rangle\\
&\quad-\left\langle\varphi,\frac\d{\d s}\Big|_{s=0}
\big(H_\phi-i\big)^{-1}\big\{u_{0,0}\big(H_\phi-i\big)^{-1}U^\phi_s
-U^\phi_su_{0,0}\big(H_\phi-i\big)^{-1}\big\}\varphi\right\rangle.
\end{align*}
For the first term, the equation
$\big(H_\phi-i\big)^{-1}=i\int_0^\infty\d r\,\e^{-r}U^\phi_r$ (valid in the strong
sense) and Lemma \ref{Lemma_form_u}(a) imply
\begin{align*}
&\frac\d{\d s}\Big|_{s=0}\big(H_\phi-i\big)^{-1}
\big\{X_f\big(H_\phi-i\big)^{-1}U^\phi_s
-U^\phi_sX_f\big(H_\phi-i\big)^{-1}\big\}\varphi\\
&=i\hspace{1pt}\frac\d{\d s}\Big|_{s=0}\int_0^\infty\d r\,\e^{-r}U^\phi_s
\big(H_\phi-i\big)^{-1}\big(U^\phi_{-s}X_fU^\phi_s-X_f\big)U^\phi_r\varphi\\
&=i\hspace{1pt}\frac\d{\d s}\Big|_{s=0}\int_0^\infty\d r\,\e^{-r}U^\phi_s
\big(H_\phi-i\big)^{-1}\int_0^{-s}\d t\,
\big(u_{0,0}\circ\phi_t\big)U^\phi_rX_\phi\varphi\\
&=-i\int_0^\infty\d r\,\e^{-r}\big(H_\phi-i\big)^{-1}u_{0,0}U^\phi_rX_\phi\varphi\\
&=i\big(H_\phi-i\big)^{-1}u_{0,0}H_\phi\big(H_\phi-i\big)^{-1}\varphi.
\end{align*}
For the second term, a direct calculation gives
\begin{align*}
\frac\d{\d s}\Big|_{s=0}\big(H_\phi-i\big)^{-1}
\big\{u_{0,0}\big(H_\phi-i\big)^{-1}U^\phi_s
-U^\phi_su_{0,0}\big(H_\phi-i\big)^{-1}\big\}\varphi
=-i\hspace{0.5pt}\big[\big(H_\phi-i\big)^{-1},u_{0,0}\big]\varphi.
\end{align*}
So, putting together the last equations, we get
\begin{align*}
&\big\langle\big(H_\phi+i\big)^{-1}\varphi,A\varphi\big\rangle
-\big\langle A\varphi,\big(H_\phi-i\big)^{-1}\varphi\big\rangle\\
&=\big\langle\varphi,
\big\{-2i\big(H_\phi-i\big)^{-1}u_{0,0}H_\phi\big(H_\phi-i\big)^{-1}
+i\hspace{0.5pt}\big[\big(H_\phi-i\big)^{-1},u_{0,0}\big]\big\}\varphi\big\rangle.
\end{align*}
Therefore, for each $t>0$ we obtain
\begin{align*}
&\big\langle\big(H_\phi+i\big)^{-1}\varphi,A_t\varphi\big\rangle
-\big\langle A_t\varphi,\big(H_\phi-i\big)^{-1}\varphi\big\rangle\\
&=\frac1t\int_0^t\d s\,\big\langle\big(H_\phi+i\big)^{-1}U^\phi_{-s}\varphi,
AU^\phi_{-s}\varphi\big\rangle-\big\langle AU^\phi_{-s}\varphi,
\big(H_\phi-i\big)^{-1}U^\phi_{-s}\varphi\big\rangle\\
&=\frac1t\int_0^t\d s\,\big\langle U^\phi_{-s}\varphi,
\big\{-2i\big(H_\phi-i\big)^{-1}u_{0,0}H_\phi\big(H_\phi-i\big)^{-1}
+i\hspace{0.5pt}\big[\big(H_\phi-i\big)^{-1},u_{0,0}\big]\big\}
U^\phi_{-s}\varphi\big\rangle\\
&=\big\langle\varphi,\big\{-2i\big(H_\phi-i\big)^{-1}c_tH_\phi
\big(H_\phi-i\big)^{-1}
+i\hspace{0.5pt}\big[\big(H_\phi-i\big)^{-1},c_t\big]\big\}\varphi\big\rangle
\end{align*}
with $c_t=\frac1t\int_0^t\d s\,(u_{0,0}\circ\phi_s)$, and the claim follows by the
density of $C^1(M)$ in $\dom(A_t)$.

(b) We know from point (a) that $\big(H_\phi-i\big)^{-1}\in C^1(A_t)$ with
\begin{align*}
&\big[i\big(H_\phi-i\big)^{-1},A_t\big]\\
&=2\big(H_\phi-i\big)^{-1}c_tH_\phi\big(H_\phi-i\big)^{-1}
-\big[\big(H_\phi-i\big)^{-1},c_t\big]\\
&=2\big(H_\phi-i\big)^{-1}c_t
+2i\big(H_\phi-i\big)^{-1}c_t\big(H_\phi-i\big)^{-1}
-\big[\big(H_\phi-i\big)^{-1},c_t\big].
\end{align*}
So, it is sufficient to show that $c_t\in C^1(A_t)$. For this, we note that
$c_t\in C^1(M)$ due to the assumption $\rho^{-1}X_f(\rho)\in C^1(M)$, Lemma
\ref{Lemma_form_u}(c) and Leibniz integral rule. Then, we use \eqref{eq_A_t} to get
for $\varphi\in C^1(M)$
\begin{align*}
&\big\langle c_t\varphi,A_t\varphi\big\rangle
-\big\langle A_t\varphi,c_t\varphi\big\rangle\\
&=\big\langle\varphi,ic_t\big(2X_f+a_tX_\phi+b_t\big)\varphi\big\rangle
-\big\langle\varphi,i\big(2X_f+a_tX_\phi+b_t\big)c_t\varphi\big\rangle\\
&=\big\langle\varphi,-i\big(2X_f(c_t)+a_tX_\phi(c_t)\big)
\varphi\big\rangle
\end{align*}
with $\big(2X_f(c_t)+a_tX_\phi(c_t)\big)$ a bounded multiplication operator, and we
note that this implies the claim because $C^1(M)$ is dense in $\dom(A_t)$.
\end{proof}

In the next proposition, we show that $H_\phi$ satisfies a strict Mourre estimate on
compact subsets of $(0,\infty)$ and $(-\infty,0)$. We use the operator
$-\ln(\lambda)A_t$ as conjugate operator on $(0,\infty)$ and the operator
$\ln(\lambda)A_t$ as conjugate operator on $(-\infty,0)$.

\begin{Proposition}[Strict Mourre estimate]\label{Proposition_Mourre}
Suppose that Assumptions \ref{ass_1}, \ref{ass_2}, and \ref{ass_3} are satisfied.
\begin{enumerate}
\item[(a)] If $I\subset(0,\infty)$ is a compact set with
$I\cap\sigma(H_\phi)\ne\varnothing$, then there exist $t>0$ and $a>0$ such that
$$
E^{H_\phi}(I)\big[iH_\phi,-\ln(\lambda)A_t\big]E^{H_\phi}(I)
\ge a\hspace{1pt}E^{H_\phi}(I).
$$
\item[(b)] If $J\subset(-\infty,0)$ is a compact set with 
$J\cap\sigma(H_\phi)\ne\varnothing$, then there exist $t>0$ and $a>0$ such that
$$
E^{H_\phi}(J)\big[iH_\phi,\ln(\lambda)A_t\big]E^{H_\phi}(J)
\ge a\hspace{1pt}E^{H_\phi}(J).
$$
\end{enumerate}
\end{Proposition}

Before the proof, we recall that the flow $\{\phi_s\}_{s\in\R}$ is ergodic. Therefore,
the spectrum of the operator $H_\phi$ in $\R\setminus\{0\}$ is not empty, and thus
there exist compact sets $I\subset(0,\infty)$ such that
$I\cap\sigma(H_\phi)\ne\varnothing$ and/or compacts sets $J\subset(-\infty,0)$ such that
$J\cap\sigma(H_\phi)\ne\varnothing$.

\begin{proof}
(a) Let $t>0$. Then, we know from Lemma \ref{Lemma_C2}(a) that
$\big(H_\phi-i\big)^{-1}\in C^1\big(-\ln(\lambda)A_t\big)$ with
$$
\big[i\big(H_\phi-i\big)^{-1},-\ln(\lambda)A_t\big]
=-2\ln(\lambda)\big(H_\phi-i\big)^{-1}c_tH_\phi\big(H_\phi-i\big)^{-1}
+\ln(\lambda)\big[\big(H_\phi-i\big)^{-1},c_t\big].
$$
This, together with \eqref{eq_resolvent}, implies that
\begin{align*}
&E^{H_\phi}(I)\big[iH_\phi,-\ln(\lambda)A_t\big]E^{H_\phi}(I)\\
&=-\big(H_\phi-i\big)E^{H_\phi}(I)\big[i\big(H_\phi-i\big)^{-1},
-\ln(\lambda)A_t\big]\big(H_\phi-i\big)E^{H_\phi}(I)\\
&=2\ln(\lambda)E^{H_\phi}(I)c_tH_\phi E^{H_\phi}(I)
-\ln(\lambda)\big(H_\phi-i\big)E^{H_\phi}(I)\big[\big(H_\phi-i\big)^{-1},c_t\big]
\big(H_\phi-i\big)E^{H_\phi}(I)\\
&=2\big(\ln(\lambda)\big)^2H_\phi E^{H_\phi}(I)
+2\ln(\lambda)E^{H_\phi}(J)\big(c_t-\ln(\lambda)\big)H_\phi E^{H_\phi}(I)\\
&\quad-\ln(\lambda)\big(H_\phi-i\big)E^{H_\phi}(I)
\big[\big(H_\phi-i\big)^{-1},c_t-\ln(\lambda)\big]\big(H_\phi-i\big)E^{H_\phi}(I)\\
&\ge a_IE^{H_\phi}(I)
+2\ln(\lambda)E^{H_\phi}(I)\big(c_t-\ln(\lambda)\big)H_\phi E^{H_\phi}(I)\\
&\quad-\ln(\lambda)\big(H_\phi-i\big)E^{H_\phi}(I)
\big[\big(H_\phi-i\big)^{-1},c_t-\ln(\lambda)\big]\big(H_\phi-i\big)E^{H_\phi}(I)
\end{align*}
with $a_I:=2\big(\ln(\lambda)\big)^2\inf(I)>0$. Furthermore, since
$\{\phi_s\}_{s\in\R}$ is uniquely ergodic, we obtain from Lemma \ref{Lemma_limit} that
$$
\lim_{t\to\infty}\big(c_t-\ln(\lambda)\big)
=\int_M\d\mu\,u_{0,0}-\ln(\lambda)
=0
$$
uniformly on $M$. Therefore, if $t>0$ is large enough, there exists $a\in(0,a_I)$ such that
$$
E^{H_\phi}(I)\big[iH_\phi,-\ln(\lambda)A_t\big]E^{H_\phi}(I)
\ge a\hspace{1pt}E^{H_\phi}(I).
$$

(b) The proof is similar to that of point (a).
\end{proof}

The next theorem is the main result of the paper.

\begin{Theorem}[Absolutely continuous spectrum]\label{Thm_spectrum}
Suppose that Assumptions \ref{ass_1}, \ref{ass_2}, and \ref{ass_3} are satisfied.
Then, $H_\phi$ has purely absolutely continuous spectrum, except at $0$, where it has
a simple eigenvalue with eigenspace $\C\cdot1$.
\end{Theorem}

\begin{proof}
We know from Lemma \ref{Lemma_C2}(b) that $H_\phi$ is of class
$C^2\big(-\ln(\lambda)A_t\big)$ for all $t>0$. Moreover, we know from Proposition
\ref{Proposition_Mourre}(a) that for each compact set $I\subset(0,\infty)$ with
$I\cap\sigma(H_\phi)\ne\varnothing$ there exist $t>0$ and $a>0$ such that
$$
E^{H_\phi}(I)\big[iH_\phi,-\ln(\lambda)A_t\big]E^{H_\phi}(I)
\ge a\hspace{1pt}E^{H_\phi}(I).
$$
Therefore, it follows from Theorem \ref{thm_absolute} that $H_\phi$ has purely
absolutely continuous spectrum in $(0,\infty)$. Since the same holds for
$(-\infty,0)$,	$H_\phi$ has purely absolutely continuous spectrum, except at $0$,
where it has a simple eigenvalue with eigenspace $\C\cdot1$ due to the ergodicity of
the flow $\{\phi_s\}_{s\in\R}$.
\end{proof}

We conclude with some remarks on Theorem \ref{Thm_spectrum}.

\begin{Remark}\label{rem_final}
(a) Under Assumptions \ref{ass_1}, \ref{ass_2}, and \ref{ass_3}, the result of Theorem
\ref{Thm_spectrum} applies, as in the case of Theorem \ref{Theorem_mixing}, to
reparametrisations of classical horocycle flows on the unit tangent bundle of compact
connected orientable surfaces of constant negative curvature.

Our regularity assumptions on the function $\rho$ are $X_f(\rho)\in C(M)$ and
$\rho^{-1}X_f(\rho)\in C^1(M)$. Therefore, Theorem \ref{Thm_spectrum} extends the
results of \cite{FU12,Tie12,Tie15} on the absolutely continuous spectrum of time
changes of the classical horocycle flows on the unit tangent bundle of compact
orientable surfaces of constant negative curvature since the function corresponding to
$\rho$ in \cite{FU12,Tie12,Tie15} is at least of class $C^3$ (in \cite{Tie12}, the function
$\rho$ is of class $C^2$, but it satisfies another additional assumption). Also, the
result of Theorem \ref{Thm_spectrum} answers the question on the regularity of the
function $\rho$ raised in \cite[Rem.~3.4]{Tie15}. The result of \cite{FU12} on
Lebesgue spectrum and the result of \cite{Tie12} for surfaces of finite volume are of
a different nature and are not covered by Theorem \ref{Thm_spectrum}.

(b) In the case of the classical horocycle flows on the unit tangent bundle of compact
connected orientable surfaces of negative curvature, L. W. Green has shown in
\cite[Thm.~B]{Gre78} that the curvature of the surface is necessarily constant if the
flow $\{\phi_s\}_{s\in\R}$ admits a $C^2$ uniformly expanding reparametrisation
$\{\widetilde\phi_s\}_{s\in\R}$ (and thus a function $\rho$ of class $C^1$).  Since it
is possible that a similar argument also applies under our regularity assumptions (an
anonymous referee signaled this to us), we prefer not to mention the
case of compact surfaces of variable negative curvature. However, we hope that in some future
the technics and the regularity assumptions presented in this manuscript could be improved in order to cover explicit examples
of compact surfaces of variable negative curvature.

(c) In Lemma \ref{Lemma_C2}, Proposition \ref{Proposition_Mourre} and Theorem
\ref{Thm_spectrum}, we stated and proved the results using only the operator $H_\phi$
under study, and not its square $(H_\phi)^2$ as we did in \cite{Tie12,Tie15}. This
allowed us to simplify the exposition in comparison to \cite{Tie12,Tie15}.
\end{Remark}


\def\cprime{$'$} \def\polhk#1{\setbox0=\hbox{#1}{\ooalign{\hidewidth
  \lower1.5ex\hbox{`}\hidewidth\crcr\unhbox0}}}
  \def\polhk#1{\setbox0=\hbox{#1}{\ooalign{\hidewidth
  \lower1.5ex\hbox{`}\hidewidth\crcr\unhbox0}}}
  \def\polhk#1{\setbox0=\hbox{#1}{\ooalign{\hidewidth
  \lower1.5ex\hbox{`}\hidewidth\crcr\unhbox0}}} \def\cprime{$'$}
  \def\cprime{$'$} \def\polhk#1{\setbox0=\hbox{#1}{\ooalign{\hidewidth
  \lower1.5ex\hbox{`}\hidewidth\crcr\unhbox0}}}
  \def\polhk#1{\setbox0=\hbox{#1}{\ooalign{\hidewidth
  \lower1.5ex\hbox{`}\hidewidth\crcr\unhbox0}}} \def\cprime{$'$}
  \def\cprime{$'$} \def\cprime{$'$}


\end{document}